\documentclass[12pt,reqno]{amsart}
\pdfoutput=1
\usepackage{xifthen}
\usepackage{tabularx}
\usepackage{adjustbox}
\usepackage{makecell}

\newtheorem*{remark*}{Remark}

\usepackage{subfig}
\newtheorem{theorem}{Theorem}[section]
\newtheorem{lemma}{Lemma}[section]

\usepackage{amsthm}

\def\Z{\mathbb Z}
\def\Q{\mathbb Q}

\def\C{\mathbb C}

\usepackage{amssymb}

\newcommand\tab[1][1cm]{\hspace*{#1}}

\usepackage{amsmath}
\usepackage{amssymb, latexsym, amsfonts, amscd, amsthm, mathrsfs, enumerate, esint}
\usepackage[usenames,dvipsnames]{color}
\usepackage[colorlinks=true,linkcolor=blue]{hyperref}
\usepackage{url}
\usepackage{bm}
\usepackage{stmaryrd}
\usepackage{hyperref}

\begin{document}

\title{An Analogue of Gauss Composition for Binary Cubic Forms}

\author{Benjamin Nativi}
\address{Department of Mathematics\\
Duke University\\
Durham, NC 27708}
\email{benjamin.nativi@duke.edu}

\begin{abstract}
Over 200 years ago, Gauss discovered a composition law on the $SL_2(\Z)$-equivalence classes of primitive binary quadratic forms. Since then, bijections of classes of binary forms have been found with ideal class groups of quadratic rings. This paper uses one such bijection given by Bhargava \cite{Bhargava1:2004}, relating classes of projective binary cubic forms to the $3$-torsion of an ideal class group, to find an explicit form for a cubic analogue of Gaussian composition.
\end{abstract}
\date{\today}
\maketitle

\section{Introduction}

In $\mathit{Disquisitiones}$ $\mathit{Arithmeticae}$, Gauss defines a fundamental law of composition on the set of binary quadratic forms, as discussed in \cite{Cohen:2000}. For two primitive binary quadratic forms $p_1$ and $p_2$ of the same discriminant $D$, their composition is any primitive binary quadratic form $P$ of discriminant $D$ such that
$$p_1(x_1,y_1)p_2(x_2,y_2) = P(X,Y),$$
where $X$ and $Y$ are integral linear combinations of the $x_1x_2$, $x_1y_2$, $y_1x_2$, and $y_1y_2$. Though the composition of two quadratic forms is not unique, the composition is well-defined on $GL_2(\Z)$-equivalence classes of quadratic forms, inducing a group structure \cite{Cox:2013}. 

Around 1838, Dirichlet discovered a useful consequence of this group structure when he found a bijection between the set of $GL_2(\Z)$-equivalence classes of binary quadratic forms and the set of ideal classes of quadratic orders. This bijection has been much discussed and extended, such as by Wood who gave a relationship between binary quadratic forms and modules for
quadratic algebras over any base ring \cite{Wood:2011}.

It seems natural to ask if we can find generalized composition laws for other sets of forms that induce group structures, and in 2004 Bhargava was able to find such composition laws for other sets of forms by relating such forms to ideal classes of quadratic orders through his study of $2\times 2 \times 2$ cubes of integers \cite{Bhargava1:2004}. Bhargava was able to further extend this work to derive composition laws from relations between sets of forms and ideal classes of cubic\cite{Bhargava2:2004}, quartic \cite{Bhargava3:2004}, and quintic orders  \cite{Bhargava4:2008}.

This paper presents a law of composition for binary \textit{cubic} forms analogous to Gauss's law of composition. As with Gauss's law of composition for binary quadratic forms, we prove that the composition of binary cubic forms is well-defined on $SL_2(\Z)$-equivalence classes.

To state the law of composition on binary cubic forms, first fix a nonzero and non-square integer $D$ congruent to 0 or 1 modulo 4. Define $\tau_D = \sqrt{D}/2$ and let 
$$
R_D = 
\begin{cases} 
\Z[\tau_D] \text{\tab\hspace{0.01in} if } D \equiv 0 \mod 4,\\ 
\Z[\frac{1}{2} + \tau_D] \text{\hspace{0.1in} if } D \equiv 1 \mod 4,
\end{cases}
$$
denote the ring of discriminant $D$ contained in the quadratic number field $\Q(\sqrt{D})$. We use the term binary cubic form to refer to a polynomial 
$$p(x,y) = a_0x^3 + 3a_1x^2y + 3a_2xy^2 + a_3y^3,$$
where $a_0, a_1, a_2, a_3 \in \Z$. Note such a polynomial could be referred to as an integral binary cubic form with triplicate central coefficients, but we shorten this to binary cubic form for conciseness. A binary cubic form is called projective if the Hessian of the cubic form is primitive, i.e. if 
$$\gcd(a_1^2-a_0a_2, a_0a_3-a_1a_2, a_2^2-a_1a_3) = 1.$$ 
We say a binary cubic form is of discriminant $D$ if 
$$D = a_0^2a_3^2 - 3a_1^2a_2^2 + 4a_1^3a_3 + 4a_0a_2^3 - 6a_0a_1a_2a_3.$$

Let 
$$\gamma = \begin{pmatrix}p & q \\ r & s\end{pmatrix} \in SL_2(\Z).$$
The group $SL_2(\Z)$ acts on the set of projective binary cubic forms by the right action
$$(f\cdot \gamma) (x,y) = f(px+qy, rx+sy) ,$$
where $f$ is a projective binary cubic form. This action preserves discriminant and gives rise to an equivalence relation on the set of projective binary cubic forms of discriminant $D$.

To define the cubic composition law, we need to consider a definition given by Bhargava as part of his bijection which we shall discuss in Section 2. Following Bhargava \cite{Bhargava1:2004}, given a projective binary cubic form $p$, define
$$p'(x,y) = c_0x^3 + 3c_1x^2y + 3c_2xy^2 + c_3y^3,$$
where
\begin{align*}
    c_0 &= \frac{1}{2}(2a_1^3-3a_0a_1a_2 + a_0^2a_3),\\
    c_1 &= \frac{1}{2}(a_1^2a_2-2a_0a_2^2 + a_0a_1a_3),\\
    c_2 &= -\frac{1}{2}(a_1a_2^2-2a_1^2a_3 + a_0a_2a_3),\\
    c_3 &= -\frac{1}{2}(2a_2^3-3a_1a_2a_3 + a_0a_3^2).
\end{align*}
Note that because all the $a_i$ are integers, $2p'$ is a binary cubic form. As shall be relevant later, there is a unique map (up to scalar multiple) 
$$S^3(S^3(V)) \to S^3(V),$$
which is $SL_2(\C)$-equivariant (where $S^3$ denotes the third symmetric power and $V = \C^2$ is the standard representation of $SL_2(\C)$). As a result of the work of Bhargava on page 238 of \cite{Bhargava1:2004}, $p\mapsto p'$ is such a map.

Now we can define our law of composition on the set of projective binary cubic forms. For two projective binary cubic forms $p_1$ and $p_2$ of discriminant $D$, define $p_1'$ and $p_2'$ as above. Define the composition of $p_1$ and $p_2$ to be any projective binary cubic form $P$ of discriminant $D$ such that 
$$P(X,Y) = p_1'(x_1,y_1)p_2(x_2,y_2) + p_1(x_1,y_1)p_2'(x_2,y_2),$$
where $X$ and $Y$ are integral linear combinations of $x_1x_2, x_1y_2, y_1x_2, y_1y_2$. Now we arrive at the main result:

\begin{theorem}
Fix $D$ a nonsquare and nonzero integer congruent to $0$ or $1$ modulo $4$. Given two projective binary cubic forms $p_1$ and $p_2$ of discriminant $D$, their composition $P$ exists and is uniquely determined up to $SL_2(\Z)$-equivalence. Further this law of composition induces a group structure on the set of $SL_2(\Z)$-equivalence classes of projective binary cubic forms in which all nontrivial elements have order $3$.
\end{theorem}

The organization of the paper is as follows. In Section 2, we will discuss Bhargava's bijection and its relation to our definition of cubic composition. We will then prove our main result, Theorem 1.1. In Section 3, we give two examples of the composition of binary cubic forms, each highlighting a different aspect of the group structure.

\newpage

\section{Cubic Gaussian Composition}

The proof of Theorem 1.1 will use an important connection between binary cubic forms and the ideal class group.  The following is a result of Bhargava \cite{Bhargava1:2004}:

\begin{theorem}
Fix $D$ a nonzero and non-square integer congruent to 0 or 1 modulo 4. Let $R_D$ be the quadratic ring of integers of discriminant $D$ and fix an orientation on $R_D$. There is a bijection between $SL_2(\Z)$-orbits of binary cubic forms $a_0x^3 + 3a_1x^2y + 3a_2xy^2 + a_3y^3$ of discriminant $D$ and the set of equivalence classes of pairs $(J,\delta)$ where $J$ is an ideal of $R_D$ and $\delta$ is an invertible element of $R_D\otimes \Q$ such that $J^3 \subset \delta \cdot R_D$ and $N(J)^3 = N(\delta)$.  
\end{theorem}

Here two pairs $(J, \delta)$ and $(J', \delta')$ are equivalent if there is an automorphism $\phi$ of $R_D$ and an element $\kappa \in R_D\otimes \Q$ such that $J' = \kappa \phi(J)$ and $\delta' = \kappa^3\phi(\delta)$.

Because we consider the oriented quadratic ring of integers, the bijection in Theorem 2.1 restricts to a bijection between \emph{projective} binary cubic forms of discriminant $D$ and pairs $(J,\delta)$ where $J$ is an \emph{invertible} ideal of $R_D$ \cite{Bhargava1:2004}.

To give the explicit bijection from Theorem 2.1, we need to discuss the two maps given by Bhargava that give the bijection. For an ideal $J \subset R_D$, we say a basis $\alpha, \beta$ for $J$ is positively oriented if 
$$N(\alpha^*\beta) = \frac{\alpha^*\beta - \beta^*\alpha}{\sqrt{D}}>0,$$
where $\alpha^*$ denotes the conjugate of $\alpha$ in $R_D$. Suppose $J = \alpha \Z + \beta \Z$ for $\alpha,\beta\in R_D$ giving a positively oriented basis for $J$ and $\delta$ with the above properties. Then we can write
\begin{align*}
    \delta^{-1}\alpha^3      &= c_0 + a_0\tau_D,\\
    \delta^{-1}\alpha^2\beta &= c_1 + a_1\tau_D,\\
    \delta^{-1}\alpha\beta^2 &= c_2 + a_2\tau_D,\\
    \delta^{-1}\beta^3       &= c_3 + a_3\tau_D.
\end{align*}
for integers $a_i$ and $c_j$ because $J^3 \subset \delta \cdot R_D$.
The map from pairs to binary cubic forms is given by
$$(\alpha \Z + \beta \Z,\delta) \mapsto a_0x^3 + 3a_1x^2y+3a_2xy^2 + a_3y^3.$$
Conversely, given a binary cubic form $p(x,y) = a_0x^3 + 3a_1x^2y+3a_2xy^2 + a_3y^3$, define $c_1$ and $c_2$ as in the definition of $p'$. The map from binary cubic forms to pairs is given by
$$p(x,y) \mapsto (J,\delta),$$
where $J = \alpha \Z + \beta \Z$ for $\alpha = c_1+a_1\tau_D$, $\beta= c_2 + a_2\tau_D$ and $\delta = \alpha\beta$. Bhargava shows that these maps are well-defined and inverses on $SL_2(\Z)$-orbits of binary cubic forms and classes of pairs $(J,\delta)$ \cite{Bhargava1:2004}. 

Now because of this bijection, it is natural to define another cubic form
$$\tilde{p}(x,y) = p'(x,y) + p(x,y)\tau_D.$$ 
In fact, this cubic form inspired the cubic law of composition. In the quadratic case, the composition of two primitive binary quadratic forms $p_1$ and $p_2$ results from taking their product. For the cubic case, instead of taking the product of two projective binary cubic forms $p_1$ and $p_2$, we can take the product of $\tilde{p_1}$ and $\tilde{p_2}$ to get
\begin{align*}
    \tilde{p_1}(x_1,y_1)\tilde{p_2}(x_2,y_2) &= p_1'(x_1,y_1)p_2'(x_2,y_2) + p_1(x_1,y_1)p_2(x_2,y_2)\frac{D}{4}\\
    &+ \left[p_1'(x_1,y_1)p_2(x_2,y_2) + p_1(x_1,y_1)p_2'(x_2,y_2)\right]\tau_D.
\end{align*}
Then as $p_1$ and $p_2$ are the coefficients of $\tau_D$ for $\tilde{p_1}$ and $\tilde{p_2}$, it is natural to define the composition of $p_1$ and $p_2$ to be a projective binary cubic form $P$ such that 
$$\tilde{p_1}(x_1,y_1)\tilde{p_2}(x_2,y_2) = \tilde{P}(X,Y).$$

We will also need another characterization of $\tilde{p}$. For a projective binary cubic form $p$, define $\alpha$, $\beta$, and $\delta$ as in the bijection. Then
\begin{align*}
    \delta^{-1}(\alpha x + \beta y)^3 &= \delta^{-1}\alpha^3 x^3 + 3\delta^{-1}\alpha^2 \beta x^2y + 3\delta^{-1}\alpha \beta^2 xy^2 + \delta^{-1}\beta^3y^3\\
    &= (c_0+a_0\tau_D)x^3 + 3(c_1+a_1\tau_D)x^2y + 3(c_2+a_2\tau_D)xy^2 +(c_3+a_3\tau_D)y^3\\
    &= p'(x,y) + p(x,y)\tau_D = \tilde{p}(x,y).
\end{align*}
Thus we have two characterizations of $\tilde{p}(x,y)$, one in terms of the coefficients of the cubic form $p$ and the other in terms of $\delta$ and the generators of the corresponding ideal. Although we used the first characterization to define the composition on projective binary cubic forms, we will rely on the second characterization to prove that this composition is well defined. We will then prove that, on equivalence classes, this composition is equivalent to taking the class corresponding to the product $(J_1J_2, \delta_1\delta_2)$ of pairs $(J_1,\delta_1)$ and $(J_2,\delta_2)$. 

We will restate our main result here before proving it. Recall that we defined the composition of $p_1$ and $p_2$ to be any projective binary cubic form $P$ of discriminant $D$ such that 
$$P(X,Y) = p_1'(x_1,y_1)p_2(x_2,y_2) + p_1(x_1,y_1)p_2'(x_2,y_2),$$
where $X$ and $Y$ are bilinear combinations of $x_1x_2, x_1y_2, y_1x_2, y_1y_2$.

\begin{theorem}
Fix $D$ a nonsquare and nonzero integer congruent to $0$ or $1$ modulo $4$. Given two projective binary cubic forms $p_1$ and $p_2$ of discriminant $D$, their composition $P$ exists and is uniquely determined up to $SL_2(\Z)$-equivalence. Further this law of composition induces a group structure on the set of $SL_2(\Z)$-equivalence classes of projective binary cubic forms in which all nontrivial elements have order $3$.
\end{theorem}

\begin{proof}
First the existence of $P$ follows easily from Bhargava's bijection in Theorem 2.1. Say that $p_1$ corresponds to $(J_1,\delta_1)$ and $p_2$ corresponds to $(J_2,\delta_2)$ where $J_1 = u_1\Z + v_1\Z$ and $J_2 = u_2\Z + v_2\Z$ are given as positively oriented bases. We can compute the product
\begin{align*}
    \tilde{p}_1(x_1,y_1)\tilde{p}_2(x_2,y_2) &= \delta_1^{-1}(u_1x_1 + v_1y_1)^3\delta_2^{-1}(u_2x_2 + v_2y_2)^3\\
    & = (\delta_1\delta_2)^{-1}(u_1u_2x_1x_2 + u_1v_2x_1y_2 + v_1u_2y_1x_2 + v_1v_2y_1y_2)^3.
\end{align*}
Because $u_1u_2$, $u_1v_2$, $v_1u_2$, and $v_1v_2$ generate the product $J_1J_2$, they can each be written as linear integral combinations of a $\Z$-linearly independent positively oriented basis $\alpha$, $\beta$ for $J_1J_2$, i.e. we can write
\begin{align*}
    u_1u_2 &= m_1\alpha + n_1\beta,\\
    u_1v_2 &= m_2\alpha + n_2\beta,\\
    v_1u_2 &= m_3\alpha + n_3\beta,\\
    v_1v_2 &= m_4\alpha + n_4\beta,
\end{align*}
for integers $m_i$ and $n_j$. Then we can rewrite
$$(\delta_1\delta_2)^{-1}(u_1u_2x_1x_2 + u_1v_2x_1y_2 + v_1u_2y_1x_2 + v_1v_2y_1y_2)^3$$
$$=(\delta_1\delta_2)^{-1}((m_1x_1x_2 + m_2x_1y_2 + m_3y_1x_2 + m_4y_1y_2)\alpha + (n_1x_1x_2 + n_2x_1y_2 + n_3y_1x_2 + n_4y_1y_2)\beta)^3.$$
Choosing 
\begin{align*}
    X &= m_1x_1x_2 + m_2x_1y_2 + m_3y_1x_2 + m_4y_1y_2\\
    Y &= n_1x_1x_2 + n_2x_1y_2 + n_3y_1x_2 + n_4y_1y_2,
\end{align*} 
we can write 
$$\tilde{p_1}(x_1,y_1)\tilde{p_2}(x_2,y_2) = (\delta_1\delta_2)^{-1}(\alpha X + \beta Y)^3.$$
Under the map from pairs $(J,\delta)$ to cubic forms given above, we know that the RHS is equal to 
$$\tilde{P}(X,Y) = P'(X,Y) + P(X,Y)\tau_D,$$
where $P(X,Y)$ is a projective binary cubic form in $X$ and $Y$ of discriminant $D$. Because the coefficient of $\tau_D$ of the LHS is 
$$p_1'(x_1,y_1)p_2(x_2,y_2) + p_1(x_1,y_1)p_2'(x_2,y_2),$$
we have 
$$P(X,Y) = p_1'(x_1,y_1)p_2(x_2,y_2) + p_1(x_1,y_1)p_2'(x_2,y_2).$$
Thus for any $p_1$ and $p_2$ we can always find an appropriate cubic form and $X$ and $Y$ so that the law of composition is defined.

Now the main result of this paper is showing the uniqueness of $P$ up to $SL_2(\Z)$-equivalence. We will do so by showing that any $P$ that satisfies the properties of the theorem must correspond to a pair $(J,\delta)$ of the same class as $(J_1J_2, \delta_1\delta_2)$. As mentioned earlier, the relation between $p$ and $p'$ is important here--we need a lemma from Bhargava \cite{Bhargava1:2004}:

\begin{lemma}
Suppose $p$ is a projective binary cubic form of discriminant $D$ and $q$ is a (not necessarily integral) binary cubic form such that $2q$ is integral. If we can write
$$p(x,y)\tau_D + q(x,y) = \delta^{-1}(\alpha x + \beta y)^3,$$
where $\alpha, \beta \in R_D$ gives a positively oriented ordered basis for an ideal $J \subset R_D$ and $\delta \in R_D \otimes \Q$ is invertible such that $J^3 \subset \delta \cdot R_D$ and $N(J)^3 = N(\delta)$, then the cubic form $q$ is uniquely determined and is equal to $p'$ as defined above.  
\end{lemma}

Now let $p_1$ and $p_2$ be projective binary cubic forms of discriminant $D$ as above. Define $p_1'$, $p_2'$, $\tilde{p}_1$, $\tilde{p}_2$ as previously and suppose that 
$$P(X,Y) = p_1'(x_1,y_1)p_2(x_2,y_2) + p_1(x_1,y_1)p_2'(x_2,y_2),$$
where $P$ is a projective binary cubic form of discriminant $D$ and $X$ and $Y$ are integral linear combinations of $x_1x_2$, $x_1y_2$, $y_1x_2$, $y_1y_2$. Suppose that $p_1$ corresponds to the pair $(J_1, \delta_1)$, that $p_2$ corresponds to the pair $(J_2, \delta_2)$, and that $P$ corresponds to the pair $(I,\delta)$, where $J_1 = u_1\Z + v_1\Z$, $J_2 = u_2\Z + v_2\Z$, and $I = \alpha\Z + \beta\Z$ are all positively oriented bases. By the lemma, because
\begin{align*}
    \tilde{p_1}(x_1,y_1)\tilde{p_2}(x_2,y_2) &= p_1'(x_1,y_1)p_2'(x_2,y_2) + p_1(x_1,y_1)p_2(x_2,y_2)\frac{D}{4}\\
    &+ \left[p_1'(x_1,y_1)p_2(x_2,y_2) + p_1(x_1,y_1)p_2'(x_2,y_2)\right]\tau_D\\
    &= p_1'(x_1,y_1)p_2'(x_2,y_2) + p_1(x_1,y_1)p_2(x_2,y_2)\frac{D}{4} + P(X,Y)\tau_D,
\end{align*}
we have that
$$P'(X,Y) = p_1'(x_1,y_1)p_2'(x_2,y_2) + p_1(x_1,y_1)p_2(x_2,y_2)\frac{D}{4},$$
so
$$\tilde{p}_1(x_1,y_1)\tilde{p}_2(x_2,y_2) = \tilde{P}(X,Y).$$
Explicitly, we can write this as
$$(\delta_1\delta_2)^{-1}(u_1u_2x_1x_2 + u_1v_2x_1y_2 + v_1u_2y_1x_2 + v_1v_2y_1y_2)^3 = \delta^{-1}(\alpha X + \beta Y)^3.$$
We know that $X$ and $Y$ can be written in terms of their linear combinations of $x_1x_2$, $x_1y_2$, $y_1x_2$, and $y_1y_2$ so that we can write
$$\alpha X + \beta Y = n_0x_1x_2 + n_1x_1y_2 + n_2y_1x_2 + n_3y_1y_2,$$
for some $n_0,n_1,n_2,n_3 \in R_D$. Now suppose the coefficients of $x_1^3x_2^3$ on both sides of the equation are nonzero. Then we have
$$(\delta_1\delta_2)^{-1}(u_1u_2)^3 = \delta^{-1}n_0^3,$$
which implies that 
$$\frac{\delta_1\delta_2}{\delta} = \left(\frac{u_1u_2}{n_0}\right)^3.$$
Letting $k = u_1u_2/n_0$, we see that $\delta_1\delta_2/\delta = k^3$ is a cube in $\Q(\sqrt{D})$. Then we can compute
$$\delta^{-1}(\alpha X + \beta Y)^3 = \delta^{-1}k^{-3}k^3(\alpha X + \beta Y)^3 = (\delta_1\delta_2)^{-1}(k\alpha X + k\beta Y)^3.$$
This amounts to changing by equivalence the pair $(I,\delta)$ to $(kI, k^3\delta)$ (with $\phi$ the trivial automorphism). However if the coefficients of $x_1^3x_2^3$ are zero, we can instead do the above steps with either $x_1^3y_2^3$, $y_1^3x_2^3$, or $y_1^3y_2^3$ because the coefficients of at least one of these will be nonzero on both sides of the equation. 

Because $k^3\delta = \delta_1\delta_2$ we have
$$(u_1u_2x_1x_2 + u_1v_2x_1y_2 + v_1u_2y_1x_2 + v_1v_2y_1y_2)^3 = (kn_0x_1x_2 + kn_1x_1y_2 + kn_2y_1x_2 + kn_3y_1y_2)^3.$$
Again considering the coefficients of $x_1^3x_2^3$ on both sides, we now have $(u_1u_2)^3 = (kn_0)^3$. 

We must now separate our discussion into two cases. First if $\Q(\sqrt{D})$ does not contain the third roots of unity, then for $a,a'\in \Q(\sqrt{D})$ the equality $a^3 = (a')^3$ implies that $a = a'$. In this case
\begin{align*}
    u_1u_2 &= kn_0,\\
    u_1v_2 &= kn_1,\\
    v_1u_2 &= kn_2,\\
    v_1v_2 &= kn_3.
\end{align*}
Second, if $\Q(\sqrt{D})$ contains the third roots of unity, then $(u_1u_2)^3 = (kn_0)^3$ implies $u_1u_2 = \zeta k n_0$ for some third root of unity $\zeta$. Clearly 
$$(kn_0x_1x_2 + kn_1x_1y_2 + kn_2y_1x_2 + kn_3y_1y_2)^3 = (\zeta kn_0x_1x_2 + \zeta kn_1x_1y_2 + \zeta kn_2y_1x_2 + \zeta kn_3y_1y_2)^3,$$
which is the same as equivalence between $(kI, k^3\delta)$ and $(\zeta kI, k^3\delta)$. Let
\begin{align*}
    f &= u_1u_2x_1x_2 + u_1v_2x_1y_2 + v_1u_2y_1x_2 + v_1v_2y_1y_2,\\
    g &= \zeta kn_0x_1x_2 + \zeta kn_1x_1y_2 + \zeta kn_2y_1x_2 + \zeta kn_3y_1y_2.
\end{align*}
We know that $f^3 = g^3$ so we can multiply by $(u_1u_2)^{-3}$ to get $(f')^3 = (g')^3$ where
\begin{align*}
    f' &= x_1x_2 + bx_1y_2 + cy_1x_2 + dy_1y_2,\\
    g' &= x_1x_2 + b'x_1y_2 + c'y_1x_2 + d'y_1y_2,
\end{align*}
for some $b,b',c,c',d,d'$ in $\Q(\sqrt{D})$. The coefficients of $x_1^3x_2^2y_2$, $x_1^2y_1x_2^3$, and $x_1^2y_1x_2^2y_2$ for $(f')^3=(g')^3$ give 
$$3b = 3b', 3c = 3c', 3d = 3d'.$$
Thus $f' = g'$ so $f=g$. 

We have in general proven that 
$$u_1u_2x_1x_2 + u_1v_2x_1y_2 + v_1u_2y_1x_2 + v_1v_2y_1y_2 = \zeta k \alpha X + \zeta k \beta Y.$$
This shows us that the generators $u_1u_2, u_1v_2, v_1u_2, v_1v_2$ of $J_1J_2$ can be written as integral combinations of $\zeta k\alpha$ and $\zeta k \beta$, the generators of $\zeta k I$ as a $\Z$-module. Thus $J_1J_2 \subset \zeta k I$. But we also know that $(J_1J_1, \delta_1\delta_2)$ and $(\zeta k I, \delta_1\delta_2)$ have the property
$$N(J_1J_2)^3 = N(\delta_1\delta_2) = N(\zeta k I)^3.$$
Because the norm of an ideal is in $\Q$, this implies
$$N(J_1J_2) = N(\zeta k I).$$
Because one ideal contains the other and their norms are equal, we have
$$J_1J_2 = \zeta k I.$$
Thus any $P$ that satisfies the law of composition corresponds to a pair $(I,\delta)$ that is in the same equivalence class as $(J_1J_2,\delta_1\delta_2)$. 

We have shown that the law of composition for projective binary cubic forms is well-defined on $SL_2(\Z)$-equivalence classes! The induced group structure and the fact that all nontrivial elements have order $3$ results from the fact that we have shown that the composition of classes of binary cubic forms corresponds to the product of the corresponding pairs.
\end{proof}

\section{Example}

We finish by giving two examples of the composition of projective binary cubic forms. In the first example, $\Q(\sqrt{D})$ contains the third roots of unity but the ideal class group is nontrivial. In the second example, the ideal class group is trivial but there are still nontrivial classes of binary cubic forms corresponding to units in $R_D$ that are not cubes.

First, for $D = -31$, the quadratic number field $\Q(\sqrt{-31})$ contains the third roots of unity. Further, it is know that the ideal class group of $R_D$ has order $3$, thus under the given law of composition the group of $SL_2(\Z)$-equivalence classes of binary cubic forms is isomorphic to $\Z/3\Z$. Let 
\begin{align*}
    p_1(x_1,y_1) &= -x_1^3 - 3x_1^2y_1 + 3x_1y_1^2 + 4y_1^3,\\
    p_2(x_2,y_2) &= x_2^3 - 6x_2^2y_2 + y_2^3,\\
    P(X,Y)       &= 7X^3 + 3X^2Y - 3XY^2.
\end{align*}
All three of these binary cubic forms are projective, of discriminant $-31$ and represent an $SL_2(\Z)$-equivalence class corresponding to a different ideal class of the ideal class group ($P$ corresponds to the trivial class). We compute
\begin{align*}
    p'_1(x_1,y_1) &= -\frac{1}{2}x_1^3 + \frac{21}{2}x_1^2y_1 + \frac{39}{2}x_1y_1^2 + y_1^3,\\
    p'_2(x_2,y_2) &= -\frac{15}{2}x_2^3 - 3x_2^2y_2 + 12 x_2y_2^2 - \frac{1}{2}y_2^3.
\end{align*}
Then the composition is
$$P(X,Y) = p_1'(x_1,y_1)p_2(x_2,y_2) + p_1(x_1,y_1)p_2'(x_2,y_2),$$
where $X = x_1x_2 + y_1x_2 - y_1y_2$ and $Y = 2x_1y_2 + 4y_1x_2 + y_1y_2$. One can verify this by checking that both the LHS and RHS are equal to
\begin{align*}
    7 & x_1^3x_2^3    + 6x_1^3x_2^2y_2      - 12x_1^3x_2y_2^2           
  + 33x_1^2y_1x_2^3   - 54x_1^2y_1x_2^2y_2  - 36x_1^2y_1x_2y_2^2 + 12x_1^2y_1y_2^3\\
  - 3 & x_1y_1^2x_2^3 - 126x_1y_1^2x_2^2y_2 + 36x_1y_1^2x_2y_2^2 + 18x_1y_1^2y_2^3
  - 29y_1^3x_2^3      - 18y_1^3x_2^2y_2     + 48y_1^3x_2y_2^2    - y_1^3y_2^3.
\end{align*}

Second, for $D = 5$, the quadratic number field $\Q(\sqrt{5})$ does not contain the third roots of unity and has a trivial ideal class group. Then the group of $SL_2(\Z)$-equivalence classes of binary cubic forms under the given composition law is isomorphic to the group of units of the ring of integers of discriminant $5$ modulo the cubes of units, which is nontrivial. Let 
\begin{align*}
    p_1(x_1,y_1) &= -x_1^3 + 3x_1^2y_1 + y_1^3,\\
    p_2(x_2,y_2) &= -3x_2^3 + 6x_2^2y_2 - 3x_2y_2^2 + y_2^3,\\
    P(X,Y)       &= -8X^3 + 15X^2Y - 9XY^2 + 2Y^3.
\end{align*}
These binary cubic forms are projective, of discriminant $5$ and each represent a different $SL_2(\Z)$-equivalence class (again $P$ represents the trivial class). We compute
\begin{align*}
    p'_1(x_1,y_1) &= \frac{3}{2}x_1^3 - \frac{3}{2}x_1^2y_1 + 3x_1y_1^2 + \frac{1}{2}y_1^3,\\
    p'_2(x_2,y_2) &= \frac{7}{2}x_2^3 - 6x_2^2y_2 + \frac{9}{2}x_2y_2^2 - \frac{1}{2}y_2^3.
\end{align*}
Then the composition is
$$P(X,Y) = p_1'(x_1,y_1)p_2(x_2,y_2) + p_1(x_1,y_1)p_2'(x_2,y_2),$$
where $X = x_1x_2 + y_1y_2$ and $Y = x_1y_2 + y_1x_2 + y_1y_2$. One can verify this by checking that both the LHS and RHS are equal to
\begin{align*}
  - 8 & x_1^3x_2^3    + 15x_1^3x_2^2y_2     - 9x_1^3x_2y_2^2     + 2x_1^3y_2^3
  + 15x_1^2y_1x_2^3   - 27x_1^2y_1x_2^2y_2  + 18x_1^2y_1x_2y_2^2 - 3x_1^2y_1y_2^3\\
  - 9 & x_1y_1^2x_2^3 + 18x_1y_1^2x_2^2y_2  - 9x_1y_1^2x_2y_2^2  + 3x_1y_1^2y_2^3
  + 2 y_1^3x_2^3      - 3y_1^3x_2^2y_2      + 3y_1^3x_2y_2^2.
\end{align*}

\section{Acknowledgements}

The author would like to thank Dr. Aaron Pollack for the project proposal and for providing guidance throughout the project. The author would also like to thank Dr. David Kraines and the Duke University Mathematics Department for supporting the project through the Program for Research for Undergraduates (PRUV 2019).


\newpage
\bibliography{references}
\bibliographystyle{alpha}

\end{document}